\documentclass [twoside,reqno,12pt] {amsart}
\usepackage[left=1in,right=1in,top=1in,bottom=1in]{geometry}

\usepackage[hidelinks]{hyperref}
\usepackage{amsfonts}
\usepackage{amssymb}
\usepackage{color}
\usepackage{graphics}
\usepackage{comment}

\newtheorem{thm}{Theorem}[section]

\newtheorem{lem}[thm]{Lemma}
\newtheorem{prop}[thm]{Proposition}

\newtheorem{defn}[thm]{Definition}
\newtheorem{rem}[thm]{Remark}

\theoremstyle{definition}

\numberwithin{equation}{section}

\renewcommand{\Re}{\mathrm{Re}}
\renewcommand{\Im}{\mathrm{Im}}

\newcommand{\C}{\mathbb{C}}

\newcommand{\n}[1]{\langle #1 \rangle}
\newcommand{\N}{\mathbb{N}}

\newcommand{\R}{\mathbb{R}}

\newcommand{\cS}{\mathcal{S}}

\newcommand{\pM}{{\partial M}}
\newcommand{\tr}{\hbox{tr}\,}
\newcommand{\scl}{\mathrm{scl}}

\def\hat{\widehat}
\def\tilde{\widetilde}
\def \bfo {\begin {eqnarray*} }
\def \efo {\end {eqnarray*} }
\def \ba {\begin {eqnarray*} }
\def \ea {\end {eqnarray*} }
\def \beq {\begin {eqnarray}}
\def \eeq {\end {eqnarray}}
\def \supp {\hbox{supp }}
\def \diam {\hbox{diam }}

\def \p {\partial}

\usepackage{empheq}

\newcommand{\LV}{\left|}
\newcommand{\RV}{\right|}
\newcommand{\LP}{\left(}
\newcommand{\RP}{\right)}

\newcommand{\LA}{\left<}
\newcommand{\RA}{\right>}
\newcommand{\confm}{c^{-\frac{n-2}{2}}}

\newcommand{\openbdy}{(0, T)\times \p M}

\newcommand{\conf}{c^{-\frac{n-2}{4}}}

\def\hat{\widehat}
\def\tilde{\widetilde}
\def \bfo {\begin {eqnarray*} }
\def \efo {\end {eqnarray*} }
\def \ba {\begin {eqnarray*} }
\def \ea {\end {eqnarray*} }
\def \beq {\begin {eqnarray}}
\def \eeq {\end {eqnarray}}
\def \supp {\hbox{supp }}
\def \diam {\hbox{diam }}

\def \p {\partial}

\usepackage{color,comment}
\newcommand{\cF}{\mathcal{F}}
\newcommand{\cO}{\mathcal{O}}

\begin{document}
\title[inverse problem for hyperbolic equation on manifolds]
{Recovery of a time-dependent potential in hyperbolic equations on conformally transversally anisotropic manifolds}

\author[Liu]{Boya Liu}
\address{B. Liu, Department of Mathematics\\
North Carolina State University, Raleigh\\ 
NC 27695, USA\\
and
Department of Mathamatics
\\
North Dakota State University\\
Fargo\\
ND 58102, USA}
\email{boya.liu@ndsu.edu}

\author[Saksala]{Teemu Saksala}
\address{T. Saksala, Department of Mathematics\\
North Carolina State University, Raleigh\\ 
NC 27695, USA}
\email{tssaksal@ncsu.edu}

\author[Yan]{Lili Yan}
\address{L. Yan, School of Mathematics, University of Minnesota, Minneapolis, MN 55455, USA}
\email{lyan@umn.edu}

\begin{abstract}
We study an inverse problem of determining a time-dependent potential appearing in the wave equation on conformally transversally anisotropic manifolds of dimension three or higher. These are compact Riemannian manifolds with boundary that are conformally embedded in a product of the real line and a transversal manifold. Under the assumption of the attenuated geodesic ray transform being injective on the transversal manifold, we prove the unique determination of time-dependent potentials from the knowledge of a certain partial Cauchy data set.

\end{abstract}
\maketitle 

\section{Introduction and Statement of Results}
\label{sec:intro}
Let $(M, g)$ be a smooth, compact, oriented Riemannian manifold of dimension $n \ge 3$ with smooth boundary $\p M$. Throughout this paper we denote $Q =(0,T)\times M^{int}$ with $0<T<\infty$, $\overline{Q}$ the closure of $Q$, and $\Sigma = \openbdy$ the lateral boundary of $Q$. We introduce the Laplace-Beltrami operator $\Delta_g$ of the metric $g$, 
and for a given smooth and strictly positive function $c(x)$ on $M$, we consider the wave operator
\begin{equation}
\label{eq:wave_conformal}
\Box_{c,g}=c(x)^{-1}\p_t^2-\Delta_g
\end{equation}
with time-independent coefficients. 

In this paper we study an inverse problem for the linear hyperbolic partial differential operator
\begin{equation}
\label{eq:potential_conformal}
\mathcal{L}_{c,g,q}=\Box_{c,g}+q(t, x), \quad (t,x)\in Q,
\end{equation}
with a time-dependent coefficient $q\in C(\overline{Q})$, called the \textit{potential}.

We shall make two geometric assumptions, of which the first one is the following:
\begin{defn}
\label{def:CTA_manifolds}
A Riemannian manifold $(M,g)$ of dimension $n\ge 3$ with boundary $\p M$ is called conformally transversally anisotropic (CTA) if $M$ is a compact subset of a manifold $\R\times M_0^{\text{int}}$ and $g= c(e \oplus g_0)$. Here $(\R,e)$ is the real line, $(M_0,g_0)$ is a smooth compact $(n-1)$-dimensional Riemannian manifold with smooth boundary, called the transversal manifold, and $c\in C^\infty(\R\times M_0)$ is a strictly positive function. 
\end{defn}

\begin{rem}
The conformal factor $c$ in Definition \ref{def:CTA_manifolds} is the same as the coefficient $c(x)$ appearing in the wave operator \eqref{eq:wave_conformal}.
\end{rem}

Examples of CTA manifolds include precompact smooth proper subsets of Euclidean, spherical, and hyperbolic spaces, see \cite{Ferreira_Kur_Las_Salo} for some more examples of CTA manifolds. The global product structure of $M$ allows us to write every point $x\in M$ as $x=(x_1,x')$, where $x_1 \in \R$ and $x'\in M_0$. In particular, the projection $\varphi(x) = x_1$ is a \textit{limiting Carleman weight}. The existence of a limiting Carleman weight implies that a conformal multiple of the metric $g$ admits a parallel unit vector field, and the converse holds for simply connected manifolds, see \cite[Theorem 1.2]{Ferreira_Kenig_Salo_Uhlmann}. The latter condition holds if and only if the manifold $(M,g)$ is locally isometric to the product of an interval and some $(n-1)$-dimensional Riemannian manifold $(M_0,g_0)$.

In addition to the product structure of the ambient space $\R \times M_0$ of the manifold $(M,g)$, we need to also assume the injectivity of certain geodesic ray transforms on the transversal manifold $(M_0,g_0)$. This type of assumption has been implemented to solve many important inverse problems on CTA manifolds, see for instance \cite{Cekic,Ferreira_Kur_Las_Salo,Krupchyk_Uhlmann_magschr,Liu_Saksala_Yan,Yan} and the references therein.

Let us now recall some definitions related to geodesic ray transforms on Riemannian manifolds with boundary.
Geodesics of $(M_0,g_0)$ can be parametrized (non-uniquely) by points on the unit sphere bundle $SM_0 = \{(x, \xi) \in TM_0: |\xi|=1\}$. We denote
\[
\p_\pm SM_0 = \{(x, \xi) \in SM_0: x \in \textcolor{red}{\p}M_0, \: \pm \langle \xi, \nu(x) \rangle > 0\}
\]
the incoming (--) and outgoing (+) boundaries of $SM_0$ corresponding to the geodesics touching the boundary. Here $\langle \cdot, \cdot \rangle$ is the Riemannian inner product of $(M_0,g_0)$, and $\nu$ is the outward unit normal vector to $\p M_0$ with respect to the metric $g_0$.

For any $(x, \xi) \in \p_-SM_0$, we let $\gamma = \gamma_{x, \xi}$ be a geodesic of $M_0$ with initial conditions $(\gamma(0), \dot{\gamma} (0)) = (x, \xi)$. Then $\tau_{\mathrm{exit}}(x, \xi)>0$ stands for the first time when $\gamma$ meets $\p M_0$, with the convention that $\tau_{\mathrm{exit}}(x, \xi) = +\infty$ if $\gamma(\tau)\in M_0^{\text{int}}$ for all $\tau>0$. We say that a unit speed geodesic segment $\gamma: [0, \tau_{\mathrm{exit}}(x, \xi)] \to M_0$, $0<\tau_{\mathrm{exit}}(x, \xi)<\infty$, is \textit{non-tangential} if $\dot{\gamma}(0)$ and $\dot{\gamma}(\tau_{\mathrm{exit}}(x, \xi))$ are non-tangential vectors to $\p M_0$, and $\gamma(\tau)\in M_0^\mathrm{int}$ for all $0<\tau<\tau_{\mathrm{exit}}(x, \xi)$.

Given a continuous function $\alpha$ on $M_0$, the attenuated geodesic ray transform of a function $f\colon M_0 \to \R$ is given by
\begin{equation}
\label{eq:geo_trans_def}
I^\alpha(f)(x, \xi) = \int_0^{\tau_{\mathrm{exit}}(x, \xi)}  \exp\bigg[\int_{0}^{t}\alpha(\gamma_{x, \xi}(s))ds\bigg] f(\gamma_{x, \xi}(t))dt, \quad (x, \xi) \in \p_-SM_0 \setminus \Gamma_-, 
\end{equation}
where $\Gamma_- = \{(x, \xi) \in \p_-SM_0: \tau_{\mathrm{exit}}(x, \xi)= +\infty\}$. The attenuated geodesic ray transform is the mathematical basis for the medical imaging method SPECT (single-photon emission computed tomography), which
is commonly used to diagnose and monitor heart problems as well as bone and brain disorders.
Inversion of an attenuated geodesic ray transform is a crucial part of solving the Calder\'on problem on CTA manifolds \cite{Ferreira_Kenig_Salo_Uhlmann}. 

The second geometric assumption we make in this paper is as follows.

\textbf{Assumption 1.} 
\label{asu:inj}
There exists $\varepsilon>0$ such that for any smooth attenuation  $\alpha$ on $M_0$ with $\|\alpha\|_{L^\infty(M_0)}<\varepsilon$, the respective attenuated geodesic ray transform $I^\alpha$ on the transversal manifold $(M_0, g_0)$ is injective over continuous functions $f$ in the sense that if $I^\alpha(f)(x, \xi)=0$ for all $(x, \xi) \in \p_-SM_0 \setminus \Gamma_-$ such that $\gamma_{x, \xi}$ is a non-tangential geodesic, then $f=0$ in $M_0$.  

Injectivity of the attenuated geodesic ray transform on \textit{simple} manifolds for small attenuations $\alpha$ was established in \cite[Theorem 7.1]{Ferreira_Kenig_Salo_Uhlmann}. A compact, simply connected Riemannian manifold with smooth boundary is said to be simple if its boundary is strictly convex, and no geodesic has conjugate points. When $\alpha=0$, injectivity of the geodesic ray transform on simple manifolds is well-known, see  \cite{Mukhometov,Sharafutdinov}.

The attenuated geodesic ray transform $I^\alpha$ is also known to be injective when some other geometric conditions are imposed.  
For instance, it was established in \cite[Theorem 29]{deHoop_Ilmavirta} that $I^\alpha$ is injective on spherically symmetric manifolds satisfying the Herglotz condition when the attenuation $\alpha$ is radially symmetric and Lipschitz continuous. The attenuation is a constant in this paper. The Herglotz condition is a special case of a manifold satisfying a convex foliation condition, and in \cite{paternain2019geodesic} the injectivity of $I^\alpha$ is verified on this type of manifolds of dimension $n\ge 3$. Some examples of manifolds satisfying the global foliation condition are the punctured Euclidean space $\R^n \setminus \{0\}$ and the torus $\mathbb{T}^n$. We refer readers to \cite[Section 2]{paternain2019geodesic} for more examples. The convex foliation condition does not forbid the existence of conjugate points in general. 

Finally, we discuss the ``measurements" considered in this paper before presenting our main result. We observe that the limiting Carleman weight $\varphi(x)$ gives us a canonical way to define the front and back faces of $\p M$ and $\p Q$. Let $\nu$ be the outward unit normal vector to $\pM$ with respect to the metric $g$. We denote $\p M_\pm = \{x \in \p M: \pm \p_\nu\varphi(x) \ge 0\}$ and $\Sigma_\pm=(0, T)\times \p M_\pm^{\text{int}}$. 
Then we define  $U=(0, T)\times U'$ and $V=(0,T)\times V'$, where $U',V'\subset \p M$ are open neighborhoods of $\p M_+$, $\p M_-$, respectively. 

The goal of this paper is to prove the unique determination of the time-dependent potential $q(t,x)$, which appears in \eqref{eq:potential_conformal}, from the following set of partial Cauchy data
\begin{equation}
\label{eq:Cauchy_data}
\mathcal{C}_{g,q} = \{(u|_U, u|_{t=T}, \p_tu|_{t=0},\p_\nu u|_V): u\in L^2(Q), \: \mathcal{L}_{c,g,q}u=0, u|_{t=0} = 0, \: \supp u|_\Sigma \subset U \}.
\end{equation}


The wellposedness of this set has been established in \cite[Section 3]{Kian_Oksanen}. From a physical perspective, as introduced in \cite{Kian_partial_data}, the inverse problem considered in this paper can be interpreted as the determination of physical properties such as the time-evolving density of an inhomogeneous medium by probing it with disturbances generated on some parts of the boundary and at initial time, and by measuring the response on some parts of the boundary and at the end of the experiment. 





We highlight that in $\mathcal{C}_{g,q}$ the Dirichlet value is measured and supported only on roughly half of the lateral boundary $U$, and the Neumann data is measured on approximately the other half of the lateral boundary $V$. Measurements are also made at the initial time $t=0$ and the end time $t=T$. 
It follows from the domain of dependence arguments given in \cite[Subsection 1.1]{Kian_partial_data} that we can only hope to recover general time-dependent coefficients in the optimal set
\[
\mathcal{D}:=\{(t, x)\in Q:  \mathrm{dist}(x, \p M)<t<T- \mathrm{dist}(x, \p M)\}
\]
when only the lateral boundary data 
\begin{equation}
\label{eq:Lat_bound_data}
\mathcal{C}_{g,q}^{\mathrm{lat}}=\{(u|_\Sigma, \p_\nu u|_\Sigma): u\in L^2(Q), \: \mathcal{L}_{g,q}u=0, \: u|_{t=0}=\p_tu|_{t=0}=0\}
\end{equation}
is given. Hence, even for a large measurement time $T>0$, global unique recovery of general time-dependent coefficients of the hyperbolic operator \eqref{eq:potential_conformal} requires information at the beginning $\{t=0\}$ and at the end $\{t=T\}$ of the measurement.

The main result of this paper is as follows.
\begin{thm}
\label{thm:main_result}
Suppose that $(M, g)$ is a CTA manifold of dimension $n \ge 3$ and that Assumption \ref{asu:inj} holds for the transversal manifold $(M_0,g_0)$. Let $T>0$ and $q_i \in C(\overline{Q})$, $i=1, 2$. If $q_1=q_2$ on $\p Q$, then $\mathcal{C}_{g, q_1} = \mathcal{C}_{g, q_2}$ implies that $q_1=q_2$ in $Q$.
\end{thm}

\begin{rem}
Theorem \ref{thm:main_result} can be viewed as an extension of \cite{Kian_partial_data} from the Euclidean space, as well as \cite{Kian_Oksanen} from CTA manifolds with a simple transversal manifold $M_0$, to general CTA manifolds. Also, Theorem \ref{thm:main_result} does not follow from our recent work \cite{Liu_Saksala_Yan}, in which the Dirichlet data was measured on the full lateral boundary $\Sigma$. The Dirichlet data in \eqref{eq:Cauchy_data} is only measured on a subset $U$ of $\Sigma$, thus the set of Cauchy data in this paper contains less data than \cite{Liu_Saksala_Yan}.

In the presence of a  damping term, one needs a remainder term decaying in its semiclassical $H^1$-norm. Our construction of such a remainder term in \cite{Liu_Saksala_Yan} is based on interior Carleman estimates for extended manifolds. Thus, it does not give us control over the support of the trace of solution of the wave equation on the boundary $\p Q$, which is required in this paper due to the condition that $\supp u|_{\Sigma} \subset U$ in the set of partial Cauchy data \eqref{eq:Cauchy_data}. In this paper, it is sufficient to construct a CGO solution whose remainder term decays in its $L^2$-norm and satisfies the aforementioned support condition. The construction of such a remainder term differs from the one presented in \cite{Liu_Saksala_Yan}, hence Theorem \ref{thm:main_result} does not fall within the scope of \cite{Liu_Saksala_Yan}.

\end{rem}

\begin{rem}
Theorem \ref{thm:main_result} states the unique determination of continuous potentials from the set of partial Cauchy data $\mathcal{C}_{g,q}$. This is attributed to the technique presented in this paper since the concentration property of Gaussian beam quasimodes (Proposition \ref{prop:limit_behavior}) requires continuity.
\end{rem}




\begin{rem}
Assumption \ref{asu:inj} of this paper is different from the literature concerning inverse problems for elliptic operators on CTA manifolds, see for instance \cite{Ferreira_Kur_Las_Salo,Krupchyk_Uhlmann_magschr}. These works assume the invertibility of the geodesic ray transform. In the case of elliptic operators, where there is only one Euclidean direction $x_1$, the authors reduced the inverse problem to the geodesic ray transform and 
recovered the Taylor expansion of the unknown function by differentiating an expression similar to \eqref{eq:Fou_geo_trans} with respect to the variable $\lambda$ at zero.  However, this approach is not applicable in our case, as the mapping $(\lambda, \beta)\mapsto -\lambda(\beta, 1)$, appearing in \eqref{eq:Fou_geo_trans}, is a diffeomorphism only if $\lambda \ne 0$. Thus, computing $\lambda$ and $\beta$-derivatives of \eqref{eq:Fou_geo_trans} at $\lambda=0$ will not give us the Taylor expansion of the unknown potential at the origin. 
\end{rem}

\subsection{Previous literature}
In this section we only review some literature concerning the recovery of time-dependent coefficients appearing in hyperbolic equations from boundary measurements. There is also a vast amount of literature about the time-independent case, which has been discussed for instance in \cite{Liu_Saksala_Yan}. 

Most of the time-dependent results rely on the use of geometric optics (GO) solutions to the hyperbolic equation. This approach was first implemented in \cite{Stefanov} to determine time-dependent coefficients of hyperbolic equations from the knowledge of scattering data by using properties of the light-ray transform. In the Euclidean setting, recovery of a time-dependent potential $q$ from the full lateral boundary data $\mathcal{C}_q^{\mathrm{lat}}$ on the infinite cylinder $\R\times \Omega$, where $\Omega$ is a bounded Euclidean domain, was established in \cite{Ramm_Sjo}. On a finite cylinder $(0, T)\times \Omega$ with $T>\mathrm{diam}(\Omega)$, it was proved in \cite{Rakesh_Ramm} that $\mathcal{C}_q^{\mathrm{lat}}$ determines $q$ uniquely in the optimal subset $\mathcal{D}$ of $(0, T)\times \Omega$. A uniqueness result for determining a general time-dependent potential $q$ from the set of partial Cauchy data $\mathcal{C}_{g,q}$ was established in \cite{Kian_partial_data}.

Turning attention to results in the Riemannian manifolds setting, global unique determination of a time-dependent potential $q$ from both full and partial boundary measurements was proved in \cite{Kian_Oksanen} on a CTA manifold $(M, g)$ with a simple transversal manifold $M_0$. In other classes of manifolds, it was recently established in \cite{Alexakis_Feiz_Oksanen_22} that a set of full Cauchy data determines the potential $q$ uniquely in Lorentzian manifolds that satisfy certain two-sided curvature bounds and some other geometric assumptions. This curvature bound was weakened in \cite{Alexakis_Feiz_Oksanen_23} near Minkowski geometry. The proof of \cite{Alexakis_Feiz_Oksanen_22} is based on a new optimal unique continuation theorem and a generalization of the Boundary Control Method, originally developed in \cite{Belishev}, to the cases when the dependence of coefficients on time is not analytic. Indeed, the Boundary Control Method, which is a powerful tool to prove uniqueness results for time-independent coefficients appearing in hyperbolic equations \cite{Belishev, Belishev_Kurylev, Katchalov_Kurylev, KKL_book, LaOk}, is not applicable to recover time-dependent coefficients in general since it relies on an application of the unique continuation theorem analogous to \cite{Tataru}, which may fail without the aforementioned real analyticity assumption, see \cite{Alinhac,Alinhac_Baoendi}.

Aside from uniqueness results concerning only the potential, there is also some literature about determining time-dependent first order perturbations appearing in hyperbolic equations from boundary measurements as well. It was established in \cite{Kian_damping} that boundary data  $\mathcal{C}_{g,q}$, with $U=\Sigma$, determines time-dependent damping coefficients and potentials uniquely in the Euclidean setting. Very recently the authors extended this result to the setting of CTA manifolds in \cite{Liu_Saksala_Yan}. 

If a full time-dependent vector field perturbation appears in the hyperbolic equation, similar to the magnetic Schr\"odinger operator, it is only possible to recover the vector field up to a space-time  differential of a test function in $Q$. A global uniqueness result was proved in \cite{Eskin} when the dependence of coefficients on the time variable is real-analytic. This analyticity assumption was removed in \cite{Salazar}, which proved a uniqueness result on an infinite cylinder $\R\times \Omega$, where $\Omega$ is a bounded domain in $\R^n$. We refer readers to \cite{Krishnan_Vashisth} for a global uniqueness result from partial Dirichlet-to-Neumann map on a finite cylinder $[0, T]\times \Omega$ with $T>\diam (\Omega)$.  In the Riemannian setting, it was established in \cite{Feizmohammadi_et_all_2019} that the lateral boundary data $\mathcal{C}_{g,q}^{\mathrm{lat}}$ determines the first and zeroth order perturbations up to the described gauge invariance on a certain non-optimal subset of $Q$. This result was obtained by reducing the problem to the inversion of the light-ray transform of the Lorentzian metric $-dt^2+g(x)$. The authors of \cite{Feizmohammadi_et_all_2019} also showed that the light-ray transform is invertible whenever the respective geodesic ray transform on the spatial manifold is invertible. To the best of our knowledge, the global (optimal) recovery of a one-form and a potential function, appearing in a hyperbolic operator, from a set of partial Cauchy data $\mathcal{C}_{g,q}$ \eqref{eq:Cauchy_data} (lateral boundary data $\mathcal{C}_{g,q}^{\mathrm{lat}}$ \eqref{eq:Lat_bound_data}) is still an open problem. 


\subsection{Outline for the proof of Theorem \ref{thm:main_result}}
The two main ingredients of the proof are the integral identity \eqref{eq:int_id_q}, which was derived in \cite{Kian_partial_data,Kian_Oksanen} from the set of partial Cauchy data $\mathcal{C}_{g,q}$, and the construction of complex geometric optics (CGO) solutions. Specifically, we shall construct a family of exponentially decaying solutions $u_1$ to the equation $\mathcal{L}_{c,g, q}^\ast u_1=0$ of the form
\[
u_1(t, x)=e^{-s(\beta t+\varphi(x))}(v_s(t,x)+r_{1}(t,x)), \quad (t, x)\in Q.
\]
On the other hand, due to the restrictions $\supp u|_{\Sigma}\subset U$ and $u|_{t=0}=0$ in $\mathcal{C}_{g,q}$, we need to construct a family of exponentially growing solutions $u_2$ to the equation $\mathcal{L}_{c ,g, q}u_2 =0$, which look like
\[
u_2(t, x)=e^{s(\beta t+\varphi(x))}(w_s(t,x)+r_{2}(t,x)), \quad (t, x)\in Q,
\]
and satisfy these two boundary conditions.
Here $s=\frac{1}{h}+i\lambda$ is a complex number, $h\in (0,1)$ is a semiclassical parameter, $\lambda\in \R$ and $\beta\in (\frac{1}{\sqrt{3}}, 1)$ are some fixed numbers, $v_s$ and $w_s$ are Gaussian beam quasimodes, $r_{1}$ and $r_{2}$ are correction terms that vanish in the limit $h\to 0$, and the function $\varphi(x)=x_1$ is a limiting Carleman weight on $M$. We choose the values of $\beta$ as above because the construction of $r_1$ relies on an application of an interior Carleman estimate \cite[Proposition 3.6]{Liu_Saksala_Yan}. This is derived from a boundary Carleman estimate \cite[Proposition 3.1]{Liu_Saksala_Yan}, which is valid for $\beta \in (\frac{1}{\sqrt{3}}, 1)$. For the construction of $r_2$, we may take $\beta \in [\frac{1}{2},1]$, see \cite[Theorem 4.1]{Kian_Oksanen}. 


Since the transversal manifold $(M_0,g_0)$ is not necessarily simple, the approach based on global CGO solutions, which was constructed in geodesic polar coordinates in \cite{Kian_Oksanen}, is not applicable under the geometric assumptions of this paper. In particular, it is not known whether one can solve the necessary eikonal and transport equations appearing in \eqref{eq:conjugated_L1} globally without the simplicity assumption. In  Proposition \ref{prop:Gaussian_beam} we construct Gaussian beam quasimodes for every non-tangential geodesic in the transversal manifold $M_0$  by using techniques originally developed in solving inverse problems for elliptic operators, see for instance \cite{Cekic,Ferreira_Kur_Las_Salo,Krupchyk_Uhlmann_magschr,Yan}, as well as \cite{Liu_Saksala_Yan} for hyperbolic operators. These quasimodes concentrate on the geodesic in the semiclassical limit $h \to 0$, as we shall explain in Proposition \ref{prop:limit_behavior}. The construction of the remainder terms $r_1$ and $r_2$ are given in Section \ref{sec:CGO_solution}. Here $r_1$ needs to have a stronger decay property, namely, being $\cO(h^{1/2})$ with respect to the semiclassical $H^1$-norm. This is achieved  with an interior Carleman estimate \cite[Proposition 3.6]{Liu_Saksala_Yan}. In order to find a remainder $r_2$ such that $u_2$ satisfies the required boundary conditions in the set of partial Cauchy data \eqref{eq:Cauchy_data}, we follow a different approach, which was developed in \cite{Kian_partial_data,Kian_Oksanen}.

To complete the proof of Theorem \ref{thm:main_result}, we shall substitute the CGO solutions \eqref{eq:CGO_v_old} and \eqref{eq:exp_grow_soln} into the integral identity \eqref{eq:int_id_q} and pass to the limit $h\to 0$. Lemma \ref{lem:rhs_est} implies that the right-hand side of \eqref{eq:int_id_q} vanishes in the limit $h\to 0$. The proof of this lemma requires a decay in the semiclassical $H^1$-norm for $r_1$, which is given in estimate \eqref{eq:est_r1}. Meanwhile, estimates \eqref{eq:estimate_v}, \eqref{eq:est_r1}, and \eqref{eq:est_r2}, in conjunction with the concentration property of Gaussian beam quasimodes Proposition (\ref{prop:limit_behavior}), yield that the left-hand side of \eqref{eq:int_id_q} converges to the attenuated geodesic ray transform involving the function $q_1-q_2$ in the limit $h\to 0$. This is the reason why we need Assumption \ref{asu:inj} to complete the proof.
%



The paper is organized as follows. We begin with the construction of Gaussian beam quasimodes in Section \ref{sec:Gaussian_beam}. In Section \ref{sec:CGO_solution} we construct both exponentially decaying and growing CGO solutions. Finally, we present the proof of Theorem \ref{thm:main_result} in Section \ref{sec:proof}.

\subsection*{Acknowledgments}
We would like to express our gratitude to Katya Krupchyk, Joonas Ilmavirta, and Hanming Zhou for their valuable discussions and suggestions. We are also very grateful to the anonymous referees for their feedback, which led to significant improvements of the paper. T.S. is partially supported by the National Science Foundation (DMS 2204997). L.Y. is partially supported by the National Science Foundation (DMS 2109199).

\section{Construction of Gaussian Beam Quasimodes}
\label{sec:Gaussian_beam}
Let $(M, g)$ be a CTA manifold given by Definition \ref{def:CTA_manifolds} and $T>0$. 
The goal of this section is to construct Gaussian beam quasimodes with desirable concentration properties. Gaussian beam quasimodes have been utilized extensively to solve inverse problems in Riemannian manifolds. We refer readers to \cite{Cekic, Ferreira_Kenig_Salo_Uhlmann, Ferreira_Kur_Las_Salo, Krupchyk_Uhlmann_magschr,Yan} for some applications in elliptic operators and \cite{Feizmohammadi_et_all_2019,KKL_book,Liu_Saksala_Yan} in hyperbolic operators. 

To streamline the construction, we first note that due to the conformal properties of the Laplace-Beltrami operator explained in \cite{Ferreira_Kur_Las_Salo}, we have
\begin{equation}
\label{eq:conformal_equivalence}
c^{\frac{n+2}{4}} (-\Delta_g)  (\conf u)= -\left(\Delta_{\tilde{g}}u+ \big(c^{\frac{n+2}{4}}\Delta_g(\conf)\big)u\right).
\end{equation}
Also, since the conformal factor $c$ is independent of the time variable $t$, we get
\begin{equation}
\label{eq:conf_damp}
c^{\frac{n+2}{4}} \p_t^2 (\conf u)=c \p_t^2 u.
\end{equation}
Thus, equations \eqref{eq:conformal_equivalence} and \eqref{eq:conf_damp} yield the following identity for the operator $\mathcal{L}_{c,g,q}$:
\begin{equation}
\label{eq:equivalence_operator}
c^{\frac{n+2}{4}}\circ \mathcal{L}_{c,g,q} \circ \conf = \mathcal{L}_{\tilde{g},  \tilde{q}},
\end{equation}
where
\begin{equation}
\label{eq:equiv_coeff}
\tilde{g}=e\oplus g_0 \quad \text{and} \quad  \tilde{q}=c(q-c^{\frac{n-2}{4}}\Delta_g(\conf)).
\end{equation}
Hence, by replacing the metric $g$ and the potential $q$ with $\tilde{g}$ and $\tilde{q}$, respectively, we may assume that the conformal factor $c=1$. In this section we shall use this assumption and consider the leading order wave operator $\Box_{e\oplus g_0}=\p_t^2-\Delta_{e\oplus g_0}$. For simplicity, let us write  $\mathcal{L}_{g,q}$ for $\mathcal{L}_{c,g,q}$ with $c=1$. Also, throughout the rest of this paper we shall denote $\mathcal{L}_{g, q}^\ast =\mathcal{L}_{g,  \overline{q}}$ the formal $L^2$-adjoint of the operator $\mathcal{L}_{g,q}$. 

We are now ready to state and prove the main result of this section.
\begin{prop}
\label{prop:Gaussian_beam}
Let $(M, g)$ be a smooth CTA manifold with boundary. Let $T>0$, and let $s=\frac{1}{h}+i\lambda$, $0<h\ll 1$, $\lambda \in \R$, and $\beta \in (0,1)$ fixed.
Let $q \in C(\overline{Q})$. Then for every unit speed non-tangential geodesic $\gamma$ of the transversal manifold $(M_0,g_0)$, there exist a one-parameter family of Gaussian beam quasimodes $v_s\in C^\infty(M_0)$ such that the estimates
\begin{equation}
\label{eq:estimate_v}
\begin{aligned}
&\|v_s\|_{L^2(M_0)}=\cO(1), 
\quad
\|e^{s(\beta t+x_1)}h^2\mathcal{L}^\ast_{g, q}e^{-s(\beta t+x_1)}v_s\|_{L^2(Q)} =\cO(h^{3/2}),
\\
&\|e^{-s(\beta t+x_1)}h^2\mathcal{L}_{g, q}e^{s(\beta t+x_1)}v_s\|_{L^2(Q)}=\cO(h^{3/2})
\end{aligned}
\end{equation}
hold as $ h\to 0$.
\end{prop}

\begin{proof}
Let $L>0$ be the length of the geodesic $\gamma=\gamma(\tau)$. By \cite[Example 9.32]{lee2012smooth}, we may embed $(M_0, g_0)$ into a larger closed manifold $(\hat{M}_0, g_0)$ of the same dimension. Also, we extend $\gamma$ as a unit speed geodesic in $\hat{M}_0$. Since $\gamma$ is non-tangential, we can choose $\varepsilon>0$ such that $\gamma(\tau)\in \hat{M}_0\setminus M_0$ and does not self-intersect for $\tau \in [-2\varepsilon, 0) \cup (L, L+2\varepsilon]$. 

Our goal is to construct Gaussian beam quasimodes near $\gamma([-\varepsilon, L+\varepsilon])$. We start by fixing a point $z_0=\gamma(\tau_0)$ on $\gamma([-\varepsilon, L+\varepsilon])$ and construct the quasimode locally near $z_0$. Let $(\tau, y) \in \Omega:=\{(\tau, y)\in \R \times \R^{n-2}: |\tau-\tau_0|<\delta,\: |y|<\delta'\}$, $\delta, \delta'>0$, be Fermi coordinates near $z_0$, see \cite[Lemma 7.4]{Kenig_Salo}. We may assume that the coordinates $(\tau, y)$ extend smoothly to a neighborhood of $\overline{\Omega}$.

We observe that near $z_0=\gamma(\tau_0)$ the trace of the geodesic $\gamma$ is given by the set $\Gamma=\{(\tau, 0): |\tau-\tau_0|<\delta\}$, and in these Fermi coordinates we have
\begin{equation}
\label{eq:g0_in_Fermi_coordinates}
g_0^{jk}(\tau,0)=\delta^{jk} \quad \text{and} \quad  \p_{y_l}g_0^{jk}(\tau, 0)=0.
\end{equation}
Hence, it follows from Taylor's theorem that for small $|y|$ we can write
\begin{equation}
\label{eq:g0_near_geo}
g_0^{jk}(\tau, y)=\delta^{jk}+\mathcal{O}(|y|^2).
\end{equation}

We first construct quasimodes $v_s$ for the conjugated operator $e^{s(\beta t+x_1)}\mathcal{L}_{g, q}^\ast e^{-s(\beta t+x_1)}$.
To that end, let us consider a Gaussian beam ansatz
\begin{equation}
\label{eq:Gaussian_beam_form}
v_s(\tau, y;h)=e^{is\Theta(\tau,y)} b(\tau, y;h).
\end{equation}
Compared to \cite{Liu_Saksala_Yan}, the quasimode $v_s$ constructed in this paper is independent of the Euclidean variables $(t,x_1)$. 

To obtain the quasimode $v_s$, we need to find a phase function $\Theta \in C^\infty(\Omega, \C)$  such that
\begin{equation}
\label{eq:property_of_phi}
\Im \Theta \ge 0, \quad \Im \Theta|_\Gamma=0, \quad \Im \Theta(\tau, y) \sim |y|^2,
\end{equation}
as well as an amplitude $b\in C^\infty(\Omega, \C)$ such that $\supp (b(\tau, \cdot)) \subset \{|y|<\delta'/2\}$. We shall achieve this by following the ideas originally presented in  \cite{KKL_book, Ralston_1982}.

Since $\Theta$ is independent of the Euclidean variables ($t,x_1)$ and $g=e\oplus g_0$, we have
\begin{equation}
\label{eq:conjugated_dt}
e^{-is\Theta}\p_t^2(e^{is\Theta} b)= 0
\end{equation}
and
\begin{equation}
\label{eq:conjugated_Laplacian}
\quad  e^{-is\Theta}(-\Delta_g)e^{is\Theta} b=-\Delta_{g_0} b-is[2\langle \nabla_{g_0}\Theta, \nabla_{g_0}b\rangle_{g_0} +(\Delta_{g_0}\Theta)b] +s^2\langle \nabla_{g_0}\Theta, \nabla_{g_0}\Theta\rangle_{g_0}b.
\end{equation}
Therefore, we obtain from \eqref{eq:conjugated_dt} and \eqref{eq:conjugated_Laplacian} that
\begin{equation}
\label{eq:conjugated_L1}
\begin{aligned}
e^{s(\beta t+x_1)} h^2\mathcal{L}^\ast_{g,q}e^{-s(\beta t+x_1)}v_s
= &
h^2e^{is\Theta}[s^2(\langle \nabla_{g_0}\Theta, \nabla_{g_0}\Theta\rangle_{g_0} - (1 - \beta^2))b
\\
&+
s(- 2i\langle \nabla_{g_0}\Theta, \nabla_{g_0}b\rangle_{g_0} - i(\Delta_{g_0}\Theta)b)
\\
&+ (-\Delta_{g_0}+\overline{q})b].
\end{aligned}
\end{equation}
From the computation above, we see that in order to verify the estimates in \eqref{eq:estimate_v}, we need to find a phase function $\Theta$ and an amplitude $b$ such that they approximately solve the eikonal and transport equations appearing on right-hand side of \eqref{eq:conjugated_L1} as multipliers of the terms $s^2$ and $s$, respectively.

Following similar arguments as in \cite{Ferreira_Kur_Las_Salo,KKL_book,Krupchyk_Uhlmann_magschr,Ralston_1977,Ralston_1982}, we aim to find $\Theta(\tau, y)\in C^\infty(\Omega, \C)$ such that
\begin{equation}
\label{eq:eikonel_eq}
\langle \nabla_{g_0}\Theta, \nabla_{g_0}\Theta\rangle_{g_0}-(1-\beta^2)=\mathcal{O}(|y|^3), \quad y\to 0,
\end{equation}
and 
\begin{equation}
\label{eq:imaginarypart_phi}
\Im \Theta\ge d|y|^2
\end{equation}
for some constant $d>0$ that depends on $\beta$. By following the arguments in \cite[Subsection 4.1.3] {Liu_Saksala_Yan}, see also \cite{Ferreira_Kur_Las_Salo, Ralston_1977, Ralston_1982}, we choose
\begin{equation}
\label{eq:phi}
\Theta(\tau, y)=\sqrt{1-\beta^2}(\tau+\frac{1}{2}H(\tau)y\cdot y),
\end{equation}
where the smooth complex-valued symmetric matrix $H(\tau)$ is the unique solution of the initial value problem for the matrix Riccati equation
\begin{equation}
\label{eq:Riccati_eq}
\dot{H}(\tau)+H(\tau)^2=F(\tau), \quad H(\tau_0)=H_0, \quad \text{ for } \tau \in \R.
\end{equation}
Here $\Im H(\tau)$ is positive definite, $H_0$ is a complex symmetric matrix such that $\Im(H_0)$ is positive definite, and $F(\tau)$ is a suitable symmetric matrix. We refer readers to \cite[Lemma 2.56]{KKL_book} for details.

We next look for an amplitude $b$ of the form 
\begin{equation}
\label{eq:amp_form_regu}
b(\tau, y; h)=h^{-\frac{n-2}{4}}b_0(\tau)\chi(y/\delta'),
\end{equation}
where $b_0\in C^\infty ([\tau_0-\delta, \tau_0+\delta])$ depends on only the travel time $\tau$ and satisfies the approximate transport equation
\begin{equation}
\label{eq:condition_b0_regu}
- 2i\langle \nabla_{g_0}\Theta, \nabla_{g_0}b_0\rangle_{g_0} - i(\Delta_{g_0}\Theta)b_0 =\cO(|y|),
\end{equation}
and the cut-off function $\chi \in C^\infty_0(\R^{n-2})$ is such that $\chi=1$ for $|y|\le 1/4$ and $\chi=0$ for $|y|\ge 1/2$.

In order to find the function $b_0$ that satisfies \eqref{eq:condition_b0_regu}, we first compute $\langle \nabla_{g_0}\Theta, \nabla_{g_0}b_0\rangle_{g_0}$. To this end, we deduce from \eqref{eq:phi} that
\begin{equation}
\label{eq:dtau_phi}
\p_\tau\Theta(\tau, y)=\sqrt{1-\beta^2}+\mathcal{O}(|y|^2).
\end{equation} 
Therefore, we get from \eqref{eq:g0_near_geo} that
\begin{equation}
\label{eq:dphidb0}
\langle \nabla_{g_0}\Theta, \nabla_{g_0}b_0\rangle_{g_0}=\sqrt{1-\beta^2}\p_\tau b_0 +\mathcal{O}(|y|^2)\p_\tau b_0.
\end{equation}

We next compute $\Delta_{g_0}\Theta$ near the geodesic $\gamma$. To that end, it follows from \eqref{eq:g0_near_geo} and \eqref{eq:phi} that 
\[
(\Delta_{g_0}\Theta)(\tau, 0) = \sqrt{1-\beta^2}\delta^{jk}H_{jk} = \sqrt{1-\beta^2} \tr H(\tau).
\]
This implies that
\begin{equation}
\label{eq:lap_phi_geo}
(\Delta_{g_0}\Theta)(\tau, y)=\sqrt{1-\beta^2} \tr H(\tau)+\mathcal{O}(|y|).
\end{equation}

To achieve \eqref{eq:condition_b0_regu}, we require that $b_0(\tau)$ satisfies
\begin{equation}
\label{eq:transport_b0}
\p_\tau b_0=-\frac{1}{2}\tr H(\tau)b_0.
\end{equation}
Hence, we have
\[
b_0(\tau) = e^{f_1(\tau)},\quad \text{where} \quad  \p_\tau f_1(\tau)= -\frac{1}{2}\tr H(\tau).
\]
Finally, we get \eqref{eq:condition_b0_regu} from \eqref{eq:dtau_phi}--\eqref{eq:transport_b0} due to the $y$-independence of $b_0$.

We next prove the estimates in \eqref{eq:estimate_v} for the quasimode
\begin{equation}
\label{eq:beam_form_regu}
v_s(\tau, y;h)=e^{is\Theta(\tau, y)}b(\tau,y; h)=e^{is\Theta(\tau, y)}h^{-\frac{n-2}{4}}b_0(\tau)\chi(y/\delta')
\end{equation}
locally in $\Omega$, where $\Omega\subset M_0$ is the domain of Fermi coordinates near the point $z_0=\gamma(\tau_0)$. 
To proceed, we shall need the following estimate for any $k\in \R$:
\begin{equation}
\label{eq:L2_power_y}
\begin{split}
\|h^{-\frac{n-2}{4}} |y|^k e^{-\frac{\Im \Theta}{h}}\|_{L^2(|y| \le \delta'/2)}
&\leq 
\|h^{-\frac{n-2}{4}} |y|^k e^{-\frac{d}{h}|y|^2}\|_{L^2(|y| \le \delta'/2)}
\\
& \leq 
\bigg(\int_{\R^{n-2}} h^{k}|z|^{2k} e^{-2d|z|^2}dz\bigg)^{1/2}
=\cO(h^{k/2}), \quad h \to 0,
\end{split}
\end{equation}
where  we  have applied estimate \eqref{eq:imaginarypart_phi} and the change of variable $z=h^{-1/2}y$.
Then it follows from \eqref{eq:imaginarypart_phi} and \eqref{eq:L2_power_y} with $k=0$ that
\begin{equation}
\label{eq:estimate_v_int}
\begin{aligned}
\|v_s\|_{L^2(\Omega)} &\le \|b_0\|_{L^\infty([\tau_0-\delta, \tau_0+\delta])} \|e^{is\Theta}h^{-\frac{n-2}{4}}\chi(y/\delta')\|_{L^2(\Omega)}
\\
&\le \mathcal{O}(1)\|h^{-\frac{n-2}{4}}e^{-\frac{d}{h} |y|^2}\|_{L^2(|y|\le \delta'/2)}=\mathcal{O}(1), \quad h\to 0.
\end{aligned}
\end{equation}

Let us  now proceed to estimate $\|e^{s(\beta t+x_1)}\mathcal{L}_{g,\overline{q}}e^{-s(\beta t+x_1)}v_s\|_{L^2(\Omega)}$, which requires estimating each term on the right-hand side of \eqref{eq:conjugated_L1}.  For the first term, by utilizing \eqref{eq:eikonel_eq},  \eqref{eq:imaginarypart_phi}, and \eqref{eq:L2_power_y} with $k=3$, we obtain
\begin{equation}
\label{eq:est_first_term}
\begin{aligned}
&h^2\|e^{is\Theta}s^2 (\langle \nabla_{g_0}\Theta, \nabla_{g_0}\Theta\rangle_{g_0} - (1 - \beta^2))b\|_{L^2(\Omega)}
\\
&=h^2\|e^{is\Theta}s^2 h^{-\frac{n-2}{4}}(\langle \nabla_{g_0}\Theta, \nabla_{g_0}\Theta\rangle_{g_0} - (1 - \beta^2))b_0\chi(y/\delta')\|_{L^2(\Omega)}
\\
&\le \cO(1) \|h^{-\frac{n-2}{4}} |y|^3 e^{-\frac{d}{h} |y|^2}\|_{L^2(|y| \le \delta'/2)} =\cO(h^{3/2}), \quad  h \to 0.
\end{aligned}
\end{equation}

We next estimate the second term on the right-hand side of \eqref{eq:conjugated_L1}. From a direct computation, we get
\[
|e^{is\Theta}|=e^{-\frac{1}{h}\Im \Theta}e^{-\lambda \Re \Theta}=e^{-\frac{\sqrt{1-\beta^2}}{2h} \Im H(\tau)y\cdot y}e^{-\lambda\sqrt{1-\beta^2}\tau}e^{-\lambda \cO(|y|^2)}.
\]
We observe that $e^{-\frac{1}{h}}=\cO(h^\infty)$. Here we say that $f=\cO(h^\infty)$ if $f=\cO(h^n)$ for every $n \in \N$. Therefore,  it follows from \eqref{eq:imaginarypart_phi} that on the support of $ \nabla_{g_0}\chi(y/\delta')$ we have
\[
|e^{is\Theta}|
\leq  
e^{-\frac{\tilde d}{h}} \quad \text{for some } \tilde d>0.
\]
Thus, by using equation \eqref{eq:condition_b0_regu}, estimate \eqref{eq:L2_power_y} with $k=1$, along with the triangle inequality, we have
\begin{equation}
\label{eq:est_second_term}
\begin{aligned}
&h^2 \|e^{is\Theta}s(-2i\langle \nabla_{g_0}\Theta, \nabla_{g_0}b\rangle_{g_0} - i(\Delta_{g_0}\Theta)b)\|_{L^2(\Omega)}
\\
&\le \cO(h) \|e^{is\Theta}h^{-\frac{n-2}{4}} [|y|\chi(y/\delta')-2i\n{\nabla_{g_0}\Theta, \nabla_{g_0}\chi(y/\delta')}_{g_0}]\|_{L^2(\Omega)}
\\
& \le \cO(h) \|h^{-\frac{n-2}{4}} |y| e^{-\frac{d}{h}|y|^2}\|_{L^2(|y| \le \delta'/2)} +\cO(e^{-\frac{\tilde{d}}{h} })
\\
&=\cO(h^{3/2}), \quad  h \to 0.
\end{aligned}
\end{equation}

Lastly, we estimate the third term on the right-hand side of \eqref{eq:conjugated_L1}. Since the amplitude $b$ is independent of $t$, it suffices to estimate  
the term involving $\Delta_g$ and the lower order term. To that end, we apply estimate \eqref{eq:L2_power_y} with $k=0$ to get
\begin{equation}
\label{eq:est_laplacian}
h^2 \|e^{is\Theta}(-\Delta_gb)\|_{L^2(\Omega)}
\le
\cO(h^2)\|h^{-\frac{n-2}{4}} e^{-\frac{d}{h} |y|^2}\|_{L^2(|y| \le \delta'/2)} 
=\cO(h^2), \quad h\to 0. 
\end{equation}

For the lower order term, it follows from \eqref{eq:L2_power_y} with $k=0$ that
\begin{equation}
\label{eq:est_first_and_zero}
h^2 \|e^{is\Theta}\overline{q}b\|_{L^2(\Omega)}=\cO(h^2), \quad h\to 0.
\end{equation}
Therefore, by combining estimates \eqref{eq:est_first_term}--\eqref{eq:est_first_and_zero}, we conclude from  \eqref{eq:conjugated_L1} that 
\begin{equation}
\label{eq:est_op_regu}
\|e^{s(\beta t+x_1)}h^2\mathcal{L}^\ast_{g, q}e^{-s(\beta t+x_1)}v_s\|_{L^2(\Omega)} = \cO(h^{3/2}), \quad h\to 0.
\end{equation}
This completes the verification of \eqref{eq:estimate_v} locally in the set $\Omega$.

To complete the construction of the quasimode $v_s$ on the transversal manifold $M_0$, we glue together the quasimodes defined along small pieces of the geodesic $\gamma$. Since $\hat{M}_0$ is compact and $\gamma(\tau):(-2\varepsilon, L+2\varepsilon)\to \hat{M}_0$ is a unit speed non-tangential geodesic that is not a loop, we get from \cite[Lemma 7.2]{Kenig_Salo} that $\gamma|_{(-2\varepsilon, L+2\varepsilon)}$ self-intersects at times $\tau_j$, where $j \in \{1,\ldots,N\}$, and
\[
-\varepsilon =\tau_0< \tau_1< \dots<\tau_N<\tau_{N+1} =L+\varepsilon.
\]
By \cite[Lemma 7.4]{Kenig_Salo}, there exists an open cover $\{(\Omega_j, \kappa_j)_{j=0}^{N+1}\}$ of $\gamma([-\varepsilon, L+\varepsilon])$ consisting of coordinate neighborhoods that have the following properties:
\begin{enumerate}
\item[(1)] $ \kappa_j(\Omega_j)=I_j\times B$, where $I_j$ are open intervals and $B=B(0, \delta')$ is an open ball in $\R^{n-2}$. Here $\delta'>0$ can be taken arbitrarily small and the same for each $\Omega_j$.
\item[(2)] $ \kappa_j(\gamma(\tau))=(\tau, 0)$ for $r \in I_j$.
\item[(3)] $\tau_j$ only belongs to $I_j$ and $\overline{I_j}\cap \overline{I_k}=\emptyset$ unless $|j-k|\le 1$.
\item[(4)] $ \kappa_j= \kappa_k$ on $\kappa_j^{-1}((I_j\cap I_k)\times B)$.
\end{enumerate}
As explained in \cite[Lemma 7.4]{Kenig_Salo}, the intervals $I_j$ can be chosen as
\[
I_0=(-2\varepsilon, \tau_1-\tilde{\delta}), \quad   I_j=(\tau_j-2\tilde{\delta}, \tau_{j+1}-\tilde{\delta}),\: j=1, \dots, N, \quad I_{N+1}=(\tau_{N+1}-2\tilde{\delta}, L+2\varepsilon)
\]
for some $\tilde{\delta}>0$ small enough. When $\gamma$ does not self-intersect, there is a single coordinate neighborhood of $\gamma|_{[-\varepsilon, L+\varepsilon]}$ such that (1) and (2) are satisfied.

We proceed as follows to construct the quasimode $v_s$. Suppose first that $\gamma$ does not self-intersect at $\tau=0$. By following the arguments from the earlier part of this proof, we find a quasimode
\[
v_s^{(0)}(\tau, y;h)=h^{-\frac{n-2}{4}}e^{is\Theta^{(0)}(\tau, y)}e^{f_1(\tau)}\chi(y/\delta')
\]
in $\Omega_0$ with some fixed initial conditions at $\tau=-\varepsilon$ for the Riccati equation \eqref{eq:Riccati_eq} determining $\Theta^{(0)}$. We next choose some $\tau_0'$ such that $\gamma(\tau_0')\in \Omega_0\cap \Omega_1$ and let
\[
v_s^{(1)}(\tau, y;h)=h^{-\frac{n-2}{4}}e^{is\Theta^{(1)}(\tau, y)}e^{f_1(\tau)}\chi(y/\delta')
\] 
be a quasimode in $\Omega_1$ by choosing the initial conditions for \eqref{eq:Riccati_eq} such that $\Theta^{(1)}(\tau_0')=\Theta^{(0)}(\tau_0')$. 
Here we have used the same function $f_1$ in both $v_s^{(0)}$ and $v_s^{(1)}$ since $f_1$ is globally defined for all $\tau\in (-2\varepsilon,L+2\varepsilon)$ and does not depend on $y$. On the other hand, since the equations determining the phase functions $\Theta^{(0)}$ and $\Theta^{(1)}$ have the same initial data in $\Omega_0$ and in $\Omega_1$, and the local coordinates $ \kappa_0$ and $ \kappa_1$ coincide on $\kappa_0^{-1}((I_0\cap I_1)\times B)$, we get $\Theta^{(1)}=\Theta^{(0)}$ in $I_0\cap I_1$. Therefore, we conclude that $v_s^{(0)}=v_s^{(1)}$ in the overlapped region $\Omega_0\cap \Omega_1$. Continuing in this way, we obtain quasimodes $v_s^{(2)}, \dots, v_s^{(N+1)}$ such that
\begin{equation}
\label{eq:quasi_corres}
v_s^{(j)}=v_s^{(j+1)} \quad \text{in} \quad \Omega_j\cap \Omega_{j+1}
\end{equation}
If $\gamma$ self-intersects at $\tau=0$, we start the construction from $v^{(1)}$ by fixing initial conditions for \eqref{eq:Riccati_eq} at $\tau=\tau_0'\in I_1$ and find $v^{(0)}$ by going backwards.

Let $\chi_j(\tau)$ be a partition of unity subordinate to $\{I_j\}_{j=1}^{N+1}$. We denote $\tilde{\chi_j}(\tau, y)=\chi_j(\tau)$ and define
\[
v_s=\sum_{j=0}^{N+1} \tilde{\chi}_jv_s^{(j)}.
\]
Then we get $v_s\in C^\infty(M_0)$. 

Let $z_1, \dots, z_R\in M_0$ be distinct self-intersection points of $\gamma$, and let $0\le \tau_1<\cdots<\tau_N$, $R\le N$, be the times of self-intersections. Let $V_j$ be a small neighborhood in $\hat{M}_0$ centered at $z_j, j=1, \dots, R$. Following the steps in \cite{Kenig_Salo}, for $\delta'$ sufficiently small we can pick a finite cover $W_1, \dots, W_S$ of remaining points on the geodesic such that $W_k\subset \Omega_{l(k)}$ for some index $l(k)$ and 
\[
\displaystyle \supp v_s \cap M_0 \subset (\cup_{j=1}^RV_j) \cup (\cup_{k=1}^SW_k).
\]
Moreover, the quasimode restricted on $V_j$ and $W_k$ is of the form
\begin{equation}
\label{eq:finite_sum_v}
v_s|_{V_j}=\sum_{l:\gamma(\tau_l)=z_j}v_s^{(l)}
\end{equation}
and 
\begin{equation}
\label{eq:quasimode_Wk}
v_s|_{W_k}=v_s^{l(k)},
\end{equation}
respectively. Since $v_s$ is a finite sum of $v_s^{(l)}$ in each case, the first and second estimate in \eqref{eq:estimate_v} follow from corresponding local considerations \eqref{eq:estimate_v_int} and \eqref{eq:est_op_regu} for each of $v_s^{(l)}$, respectively. 


We next construct a Gaussian beam quasimode for the operator $e^{-s(\beta t+x_1)}\mathcal{L}_{g,q}e^{s(\beta t+x_1)}$ of the form
$
w_s(\tau, y;h)=e^{is\Theta(\tau, y)}B(\tau, y;h)
$
with the same phase function $\Theta \in C^\infty(\Omega, \C)$ that satisfies \eqref{eq:property_of_phi}, and the amplitude $B (\tau, y)\in C^\infty(\Omega)$ is supported near $\Gamma$.

By similar computations as in \eqref{eq:conjugated_L1}, we have
\begin{equation}
\label{eq:conjugated_L2}
\begin{aligned}
e^{-s(\beta t+x_1)}\mathcal{L}_{g, q}e^{s(\beta t+x_1)}w_s
= &e^{is\Theta}[s^2(\langle \nabla_{g_0}\Theta, \nabla_{g_0}\Theta\rangle_{g_0}-(1-\beta^2))B
\\
&+s(-2i\langle \nabla_{g_0}\Theta, \nabla_{g_0}B\rangle_{g_0} -i(\Delta_{g_0}\Theta)B)
\\
&+(-\Delta_{g_0}+q)B].
\end{aligned}
\end{equation}
Notice that the eikonal equation and transport equation for $B$ coincide with the respective corresponding equation for $b$ in \eqref{eq:conjugated_L1}. Therefore, we get $B(\tau, y;h)=b(\tau, y;h)$. Furthermore, since the phase function $\Theta$ is the same for both $v_s$ and $w_s$, we have $w_s=v_s$. Finally, we obtain the third estimate in \eqref{eq:estimate_v} by arguing similarly as in the verification of the second estimate in \eqref{eq:estimate_v}. This completes the proof of Proposition \ref{prop:Gaussian_beam}.
\end{proof}

We want the Gaussian beam quasimodes to concentrate along the geodesic as $h\to 0$. By following the same arguments as in the proof of \cite[Proposition 3.1]{Ferreira_Kur_Las_Salo}, we have the following result.
\begin{prop}
\label{prop:limit_behavior}
Let $s=\frac{1}{h}+i\lambda$, $0<h\ll 1$, $\lambda \in \R$ fixed, and $\beta \in (\frac{1}{\sqrt{3}},1)$. Let  $\gamma\colon[0, L]\to M_0$ be a non-tangential geodesic in $(M_0,g_0)$ as in Proposition \ref{prop:Gaussian_beam}. 
Let $v_s$ be the quasimode from Proposition \ref{prop:Gaussian_beam}. Then for each function $\psi \in C(M_0)$ and $(t',x_1')\in  [0,T]\times \R$ we have
\begin{equation}
\label{eq:limit_prod_vw}
\lim_{h\to 0} \int_{\{t'\}\times \{x_1'\}\times M_0} |v_s|^2\psi dV_{g_0} = \int_{0}^L e^{-2\sqrt{1-\beta^2} \lambda \tau} \psi(\gamma(\tau))d\tau.
\end{equation}
\end{prop}

\section{Construction of Complex Geometric Optics Solutions}
\label{sec:CGO_solution}



In this section we construct a family of exponentially decaying solutions $u_1\in H^1(Q)$ of the form
\[
u_1(t, x)=e^{-s(\beta t+\varphi(x))}(v_s(x')+r_{1}(t,x)), \quad (t, x)\in Q,
\]
as well as a family of exponentially growing solutions $u_2\in H_{\Box_{c,g}}(Q):=\{u\in L^2(Q): \Box_{c,g}u\in L^2(Q)\}$ given by
\[
u_2(t, x)=e^{s(\beta t+\varphi(x))}(v_s(x')+r_{2}(t,x)), \quad (t, x)\in Q,
\]
satisfying $\supp u_2|_{\Sigma}\subset U$ and $u_2|_{t=0}=0$. Here $s=\frac{1}{h}+i\lambda$ with $\lambda\in \R$ fixed, $\varphi(x)=x_1$ is a limiting Carleman weight, $v_s$ is the Gaussian beam quasimodes given in Proposition \ref{prop:Gaussian_beam}, and $r_{j}$, $j=1,2$, are correction terms that vanish as $h\to 0$. These two types of solutions will play different roles in the proof of Theorem \ref{thm:main_result}, and the proofs for their existence are also somewhat different. The construction of the exponentially decaying solution $u_1$ follows from an interior Carleman estimate \cite[Proposition 3.6]{Liu_Saksala_Yan}. For the exponentially growing solution $u_2$, we shall follow the approach introduced in \cite{Kian_partial_data,Kian_Oksanen}. 

In the following proposition, which shows the existence of exponentially decaying solutions $u_1$, we equip $Q$ with a semiclassical Sobolev norm
\[
\|u\|^2_{H^1_\scl(Q)}=\|u\|^2_{L^2(Q)}+\|h\p_t u\|^2_{L^2(Q)}+\|h\nabla_g u\|^2_{L^2(Q)}.
\]
\begin{prop}
\label{prop:existence_decay_soln}
Let $q \in C(\overline{Q})$, $\beta\in (\frac{1}{\sqrt{3}}, 1)$, and let $s=\frac{1}{h}+i\lambda$ with $\lambda \in \R$ fixed. For all $h>0$ sufficiently small, there exists a solution $u_1\in H^1(Q)$ to the equation $\mathcal{L}_{c,g, q}^*u_1=0$ of the form
\begin{equation}
\label{eq:CGO_v_old}
u_1=e^{-s(\beta t+x_1)} \conf (v_s+r_1),
\end{equation}
where $v_s\in C^\infty(M_0)$ is the Gaussian beam quasimode given in Proposition \ref{prop:Gaussian_beam}, and $r_1\in H^1_{\scl}(Q^\mathrm{int})$ satisfies the estimate
\begin{equation} 
\label{eq:est_r1}
\|r_1\|_{H^1_\scl(Q^\mathrm{int})}= \cO(h^{1/2}), \quad h\to 0. 
\end{equation}
\end{prop}
\begin{proof}
Since $(M, g)$ is a CTA manifold, the computations in Section \ref{sec:Gaussian_beam} yield 
\[
c^{\frac{n+2}{4}}\circ \mathcal{L}_{c,g,q} \circ \conf = \mathcal{L}_{\tilde{g}, \tilde{q}},
\]
where $\tilde g=e\oplus g_0$ and $\tilde{q}=c(q-c^{\frac{n-2}{4}}\Delta_g(\conf))$. Hence, we see that if $\tilde{u}$ is a solution to the equation $ \mathcal{L}^\ast_{\tilde{g}, \tilde{q}} \tilde u=0$, then the function $u=\conf \tilde u$ satisfies $\mathcal{L}^\ast_{c,g,q} u=0$. Thus, it suffices to look for solutions to the equation $ \mathcal{L}^\ast_{\tilde{g}, \tilde{q}} \tilde u=0$ of the form
$
\tilde{u}=e^{-s(\beta t+x_1)}(v_s+r_1).
$
This is equivalent to finding a function $r_1$ that solves the equation
\begin{equation}
\label{eq:equation_for_r_1}
e^{s(\beta t+x_1)}h^2\mathcal{L}_{\tilde{g}, \tilde{q}}^\ast e^{-s(\beta t+x_1)}r_1=-e^{s(\beta t+x_1)}h^2\mathcal{L}_{\tilde{g}, \tilde{q}}^\ast e^{-s(\beta t+x_1)}v_s.
\end{equation}

From here we use estimate \eqref{eq:estimate_v} and apply an interior Carleman estimate \cite[Proposition 3.6]{Liu_Saksala_Yan} to deduce that there exists a function $r_1\in H^1_{\scl}(Q^\mathrm{int})$ that solves \eqref{eq:equation_for_r_1}  and satisfies estimate \eqref{eq:est_r1}.
Lastly, we would like to recall that the interior Carleman estimate we utilized in this proof needs the required assumption for the parameter $\beta$. This completes the proof of Proposition \ref{prop:existence_decay_soln}.
\end{proof}


We now turn to the construction of exponentially growing solutions $u_2$ vanishing on part of $\p Q$. We emphasize that the earlier construction of $u_1$ requires an extension of the domain due to the interior Carleman estimate. Therefore, we have no control over the traces of the solutions to the wave equation on $\p Q$ if we consider solutions on the extended domain. Thus, we need to utilize a different approach.

%
For every $\varepsilon>0$ we set
\[
\p M_{\varepsilon,-}=\{x\in \p M: \p_\nu \varphi(x) < -\varepsilon\}, \quad 
\p M_{\varepsilon,+}=\{x\in \p M: \p_\nu \varphi(x) \ge -\varepsilon\},
\]
and $\Sigma_{\varepsilon, \pm }=(0,T)\times \p M_{\varepsilon, \pm}$. 
To find $u_2$, we use the following result, which was originally proved in \cite[Theorem 5.4]{Kian_Oksanen}. For the convenience of the reader, we re-prove this result.
\begin{prop}
\label{prop:existence_exp_grow}
Let $q\in C(\overline{Q})$, $\beta \in [\frac{1}{2}, 1]$, and let $s=\frac{1}{h}+i\lambda$ with $\lambda \in \R$ fixed. For all $h>0$ small enough, the initial boundary value problem 
\begin{equation}
\label{eq:ini_bound_prob_vanish}
\begin{cases}
\mathcal{L}_{c,g,q}u_2 = 0 \quad \text{in}\quad Q,\\
u_2(0,x) = 0 \quad \text{in}\quad M,\\
u_2 = 0 \quad \text{on}\quad \Sigma_{\varepsilon,-},
\end{cases} 
\end{equation} 
admits a solution $u_2\in H_{\Box_{c,g}}(Q)$  of the form
\begin{equation}
\label{eq:exp_grow_soln}
u_2=e^{s(\beta t+x_1)}\conf(v_s+r_2).
\end{equation}
Here $v_s\in C^\infty(M_0)$ is the Gaussian beam quasimode given in Proposition \ref{prop:Gaussian_beam}, and $r_2\in L^2(Q)$ satisfies the estimate
\begin{equation}
\label{eq:est_r2}
\|r_2\|_{L^2(Q)}=\cO(h^{1/2}), \quad h\to 0.
\end{equation}
\end{prop}

Since $U' \subset \p M$ is an open neighborhood of $\pM_+=\{x\in \p M:  \p_\nu \varphi(x) \geq  0\}$, we can choose $\varepsilon>0$ sufficiently small such that $\Sigma_{\varepsilon,+}\subset U=(0,T)\times U'$. Thus, the claims $\supp u_2|_{\Sigma}\subset U$ and $u_2|_{t=0}=0$ follow from Proposition \ref{prop:existence_exp_grow}.

\begin{proof}[Proof of Propostion \ref{prop:existence_exp_grow}]

Arguing similarly as in the proof of Proposition \ref{prop:existence_decay_soln}, we may assume that the conformal factor $c = 1$ and look for solutions to the equation $ \mathcal{L}_{\tilde{g}, \tilde{q}} \tilde u=0$ of the form
$
\tilde{u}=e^{s(\beta t+x_1)}(v_s+r_2)
$
such that $\tilde{u} (0,x) = 0$ in $M$ and $\tilde u = 0$ on $\Sigma_{\varepsilon,-}$. Here $\tilde{g}$ and $\tilde{q}$ are given by \eqref{eq:equiv_coeff}.
This is equivalent to finding a function $r_2$ that satisfies
\begin{equation}
\label{eq:equation_for_r_2}
\begin{cases}
e^{-s(\beta t+x_1)}h^2\mathcal{L}_{\tilde{g}, \tilde{q}} e^{s(\beta t+x_1)}r_2=-e^{-s(\beta t+x_1)}h^2\mathcal{L}_{\tilde{g}, \tilde{q}} e^{s(\beta t+x_1)}v_s = : f , \quad \text{in}\quad Q,\\
r_2(0,x) = -v_s(x') = : r_0 \quad \text{in}\quad M,\\
r_2(t,x) = -v_s(x') = :r_- \quad \text{on}\quad \Sigma_{\varepsilon,-}.
\end{cases} 
\end{equation}
This can be achieved by solving the following initial boundary value problem
\begin{subequations}
\begin{empheq}[left=\empheqlbrace]{align}
&e^{-\frac{1}{h}(\beta t+x_1)}\mathcal{L}_{\tilde{g}, \tilde{q}} e^{\frac{1}{h} (\beta t +x_1)}\tilde r=-e^{i\lambda(\beta t +x_1)} e^{-s(\beta t+x_1)}\mathcal{L}_{\tilde{g}, \tilde{q}} e^{s(\beta t+x_1)}v_s = : \tilde f \quad \text{in}\quad Q, \label{eq:eq_tilde_r}
\\
&\tilde r(0,x) = -e^{i\lambda x_1} v_s(x') = :\tilde r_0 \quad \text{in}\quad M, 
\label{eq:initial_cond_tilde_r}
\\
&\tilde r(t,x) = - e^{i\lambda(\beta t + x_1)}v_s(x') \psi(x) = :\tilde r_- \quad \text{on}\quad \Sigma_{-}, \label{eq:lateral_cond_tilde_r}
\end{empheq}
\end{subequations}
and setting $r_2:=e^{-i\lambda(\beta t +x_1)}\tilde{r}$.
Here $\psi\in C^\infty_0(M)$ is a cut-off function such that $0\le \psi\le 1$, $\supp \psi \cap \p M \subset \p M_{\varepsilon/2,-}$, and $\psi = 1$ on $\p M_{\varepsilon,-}$. 

In order to verify the existence of $\tilde r$, we need to derive a boundary Carleman estimate for the operator $\mathcal{L}_{\tilde{g},\tilde{q}}$. To that end, let us introduce the space 
\[
\mathcal{D} := \left\{v\in C^\infty(\overline{Q}): v|_{\Sigma} = v|_{t = T} = \p_t v |_{t = T} = v |_{t = 0} = 0\right\}. 
\]
By replacing $t$ by $T-t$ and $x_1$ by $-x_1$ in the boundary Carleman estimate \cite[Lemma 4.2]{Kian_Oksanen}, we see that following estimate for the wave operator $\Box_{\tilde{g}}$ 
\begin{equation}
\label{eq:bdy_Car_est_1}
\begin{aligned}
&h^{1/2}\|\p_t u(0,\cdot)\|_{L^2(M)} + h^{1/2} \| \sqrt{\LV \p_\nu \varphi \RV}\p_\nu u  \|_{L^2(\Sigma_-)}+\|u\|_{L^2(Q)} 
\\
& \le \mathcal{O}(h) \| e^{-\frac{1}{h}(\beta t +x_1)}\Box_{\tilde{g}} e^{\frac{1}{h}(\beta t +x_1)} u \|_{L^2(Q)} + \mathcal{O}(h^{1/2})\|\sqrt{\p_\nu\varphi}\p_\nu u\|_{L^2(\Sigma_+)}
\end{aligned}
\end{equation}
is valid for any $u\in \mathcal{D}$.

To establish a boundary Carleman estimate for the operator $\mathcal{L}_{\tilde{g}, \tilde q}$, we first apply the triangle inequality to obtain
\[
\| e^{-\frac{1}{h}(\beta t +x_1)}\Box_{\tilde{g}} e^{\frac{1}{h}(\beta t +x_1)} u \|_{L^2(Q)} \le \| e^{-\frac{1}{h}(\beta t +x_1)}\mathcal{L}_{\tilde{g},\tilde{q}} e^{\frac{1}{h}(\beta t +x_1)} u \|_{L^2(Q)} + \|\tilde{q}u\|_{L^2(Q)}.
\]
In particular, we have
\[
\|\tilde{q}u\|_{L^2(Q)}\le \|\tilde{q}\|_{L^\infty(Q)}\|u\|_{L^2(Q)}.
\]
Therefore, by absorbing the term $h\|\tilde{q}\|_{L^\infty(Q)}\|u\|_{L^2(Q)}$ into the left-hand side of \eqref{eq:bdy_Car_est_1}, we obtain the following boundary Carleman estimate for the conjugated operator $\mathcal{L}_{\tilde g, \tilde q}$ 
\begin{equation}
\label{eq:bdy_Car_est_2}
\begin{aligned}
&h^{1/2}\|\p_t u(0,\cdot)\|_{L^2(M)} + h^{1/2} \| \sqrt{\LV \p_\nu \varphi \RV}\p_\nu u  \|_{L^2(\Sigma_-)}+\|u\|_{L^2(Q)} 
\\
& \le \mathcal{O}(h) \| e^{-\frac{1}{h}(\beta t +x_1)}\mathcal{L}_{\tilde{g},\tilde{q}} e^{\frac{1}{h}(\beta t +x_1)} u \|_{L^2(Q)} + \mathcal{O}(h^{1/2})\|\sqrt{\p_\nu\varphi}\p_\nu u\|_{L^2(\Sigma_+)},
\end{aligned}
\end{equation}
which holds for any function  $u\in \mathcal{D}$.

Let us recall the following estimates, which follow immediately from Proposition \ref{prop:Gaussian_beam}:
\begin{equation}\label{eq:tilde_f_r}
\|\tilde f\|_{L^2(Q)} = \mathcal{O}(h^{-1/2}), \quad \|\tilde r_0\|_{L^2(M)} = \mathcal{O}(1).
\end{equation}
Also, by utilizing the same arguments as in the proof of \cite[Lemma 5.1]{Liu_Saksala_Yan}, we obtain the estimate $\|v_s\|_{L^2(\Sigma_{\varepsilon/2,-})} = \mathcal{O}(1)$. Furthermore, since $\p_\nu \varphi<-\frac{\varepsilon}{2}$ in the set $\Sigma_{\varepsilon/2, -}$, the estimate $|\p_\nu \varphi|^{-1/2} =\cO(\varepsilon^{-1/2})$ holds in $\Sigma_{\varepsilon/2, -}$. Thus,
\begin{equation}
\label{eq:tilde_r_boudary}
\||\p_\nu \varphi|^{-1/2}\tilde r_-\|_{L^2{(\Sigma_-)}}
\le \||\p_\nu \varphi|^{-1/2} e^{i\lambda(\beta t + x_1)}v_s\|_{L^2{(\Sigma_{\varepsilon/2,-})}}
= \mathcal{O}(\varepsilon^{-1/2}).
\end{equation}

We next follow the proof of \cite[Lemma 5.1]{Kian_partial_data} closely  to verify the existence of $\tilde r\in H_{\Box_{\tilde{g}}}(Q)$ satisfying \eqref{eq:eq_tilde_r}--\eqref{eq:lateral_cond_tilde_r}. 
To that end, we introduce the space
\[
\mathcal{M} := \left\{
(e^{\frac{1}{h}(\beta t +x_1)}\mathcal{L}^\ast_{\tilde g, \tilde{q}} e^{-\frac{1}{h}(\beta t +x_1)}u, \p_\nu u|_{\Sigma_+}): u\in \mathcal{D}
\right\}
\]
and equip it with the norm
\[
\|(g_1,g_2)\|_{\mathcal{M}} = \|g_1\|_{L^2(Q)} 
+\|h^{-1/2}|\p_\nu \varphi|^{1/2}g_2\|_{L^2(\Sigma_+)}.
\]
Then $\mathcal{M}$ can be viewed as a subspace of $L^2(Q)\times L_{\varphi,h,+}^2(\Sigma_+)$, where
\[
L_{\varphi,h,+}^2(\Sigma_+) : = \{u: \|h^{-1/2}|\p_\nu\varphi|^{1/2}  u \|_{L^2(\Sigma_+)}= \mathcal{O}(1)\}.
\]
We shall construct a bounded linear functional on $\mathcal{M}$ and use a standard Hahn-Banach argument to prove the existence of $\widetilde{r}$, and the conditions \eqref{eq:initial_cond_tilde_r} and \eqref{eq:lateral_cond_tilde_r} will follow from integration by parts and density argument.

If $u\in \mathcal{D}$, by the boundary Carleman estimate \eqref{eq:bdy_Car_est_2} and the Cauchy-Schwartz inequality, we have
\begin{align*}
&\left|\LA u,\tilde{f}\RA_{L^2(Q)} - \LA \p_tu(0,\cdot), \tilde r_0 \RA_{L^2(M)} -\LA \p_\nu u,\tilde r_-\RA_{L^2(\Sigma_-)}\right|
\\
&\le \|u\|_{L^2(Q)}\|\tilde{f}\|_{L^2(Q)} + \|\p_tu(0,\cdot)\|_{L^2(M)}\|\tilde r_0\|_{L^2(M)} + \||\p_\nu \varphi|^{1/2}\p_\nu u\|_{L^2(\Sigma_-)}\||\p_\nu \varphi|^{-1/2}\tilde r_-\|_{L^2{(\Sigma_-)}}
\\
&\le \mathcal{O}(1) \|(e^{\frac{1}{h}(\beta t +x_1)}\mathcal{L}^*_{\tilde g, \tilde q} e^{-\frac{1}{h}(\beta t +x_1)}u, \p_\nu u|_{\Sigma_+})\|_{\mathcal{M}}
\\
& \qquad \quad \times \LP h\|\tilde{f}\|_{L^2(Q)} + h^{1/2}\|\tilde r_0\|_{L^2(M)} + h^{1/2}\||\p_\nu \varphi|^{-1/2}\tilde r_-\|_{L^2{(\Sigma_-)}}\RP.
\end{align*}
Hence, we may define a linear functional $\mathcal{S}$ on $\mathcal{M}$ by setting 
\[
\mathcal{S}(e^{\frac{1}{h}(\beta t +x_1)}\mathcal{L}^\ast_{\tilde g, \tilde{q}} e^{-\frac{1}{h}(\beta t +x_1)}u, \p_\nu u|_{\Sigma_+}) : = \LA u,\tilde{f}\RA_{L^2(Q)} - \LA \p_tu(0,\cdot), \tilde r_0 \RA_{L^2(M)} -\LA \p_\nu u,\tilde r_-\RA_{L^2(\Sigma_-)}, \: u\in \mathcal{D}.
\]

By the Hahn-Banach theorem, we can extend the operator $\mathcal{S}$ to a continuous linear form $\tilde{\mathcal{S}}$ on $L^2(Q)\times L^2_{\varphi,h}(\Sigma_+)$ without increasing the norm. Hence, it follows from estimates \eqref{eq:tilde_f_r} and \eqref{eq:tilde_r_boudary} that
\begin{equation}
\|\tilde{\mathcal{S}}\| = \|\mathcal{S}\| \le \mathcal{O}(1)\LP h\|\tilde{f}\|_{L^2(Q)} + h^{1/2}\|\tilde r_0\|_{L^2(M)} + h^{1/2}\||\p_\nu \varphi|^{-1/2}\tilde r_-\|_{L^2{(\Sigma_-)}}\RP = \mathcal{O}(h^{1/2}).
\end{equation}
Thus, by the Riesz representation theorem, there exists $(\tilde r,\tilde r_+)\in L^2(Q)\times L_{\varphi,h,-}^2(\Sigma_+)$
such that 
\[
\|\tilde r\|_{L^2(Q)} + \|h^{1/2}|\p_\nu\varphi|^{-1/2}  \tilde r_+\|_{L^2(\Sigma_+)} = \|\cS\| =\mathcal{O}(h^{1/2}),
\]
and
\[
\tilde{\mathcal{S}}(g_1,g_2) = \LA g_1,\tilde r \RA_{L^2(Q)} + \LA  g_2 ,\tilde r_+\RA_{L^2(\Sigma_+)},\quad (g_1,g_2)\in L^2(Q)\times L_{\varphi,h,+}^2(\Sigma_+).
\]
Here the space $L_{\varphi,h,-}^2(\Sigma_+)$ is given by
\[
L_{\varphi,h,-}^2(\Sigma_+) : = \{u: \|h^{1/2}|\p_\nu\varphi|^{-1/2}  u \|_{L^2(\Sigma_+)}= \mathcal{O}(1) \}.
\]
Therefore, for all functions $u\in \mathcal{D}$, we get 
\begin{equation}\label{eq:equality_u_inD}
\begin{aligned}
&\LA e^{\frac{1}{h}(\beta t +x_1)}\mathcal{L}^\ast_{\tilde g, \tilde{q}} e^{-\frac{1}{h}(\beta t +x_1)}u, \tilde r\RA_{L^2(Q)} + \LA  \p_\nu u|_{\Sigma_+} , \tilde r_+\RA_{L^2(\Sigma_+)} 
\\
&=\LA u,\tilde{f}\RA_{L^2(Q)} - \LA \p_tu(0,\cdot), \tilde r_0 \RA_{L^2(M)} -\LA \p_\nu u|_{\Sigma_-},\tilde r_-\RA_{L^2(\Sigma_-)}.
\end{aligned}
\end{equation}

We now verify that $\tilde r$ satisfies equations \eqref{eq:eq_tilde_r}--\eqref{eq:lateral_cond_tilde_r}. By taking $u\in C^\infty_0(Q)$ in \eqref{eq:equality_u_inD} and using the fact that $C^\infty_0(Q)$ is dense in $L^2(Q)$, we see that equation \eqref{eq:eq_tilde_r} holds. Furthermore, \eqref{eq:eq_tilde_r} implies $\tilde r \in H_{\Box_{\tilde g}}(Q)$. Thus, by \cite[Proposition A.1]{Kian_partial_data}, we can define the trace $\tilde r|_{\Sigma}\in H^{-3}(0,T;H^{-1/2}(\p M))$ and $\tilde r|_{t = 0}\in H^{-2}(M)$.
Furthermore, the density result \cite[Theorem A.1]{Kian_partial_data}, in conjunction with integration by parts, implies that for all $u\in \mathcal{D}$, we have 
\[
\begin{aligned}
&\LA e^{\frac{1}{h}(\beta t +x_1)}\mathcal{L}^\ast_{\tilde g, \tilde{q}} e^{-\frac{1}{h}(\beta t +x_1)}u, \tilde r\RA_{L^2(Q)} 
+\LA  \p_\nu u|_{\Sigma_+} , \tilde r|_{\Sigma_+}\RA_{H^{3}(0,T;H^{1/2}(\Sigma_+)),H^{-3}(0,T;H^{-1/2}(\Sigma_+))} 
\\
&=\LA u,\tilde{f}\RA_{L^2(Q)} 
- \LA \p_tu(0,\cdot), \tilde r|_{t = 0} \RA_{H^{2}(M), H^{-2}(M)} 
-\LA \p_\nu u,\tilde r|_{\Sigma_-}\RA_{H^{3}(0,T;H^{1/2}(\Sigma_-)),H^{-3}(0,T;H^{-1/2}(\Sigma_-))}.
\end{aligned}
\]

Finally, we compare the equality above with \eqref{eq:equality_u_inD} and take arbitrary $u\in \mathcal{D}$ to conclude that $\tilde r = \tilde r_0$ and $\tilde r|_{\Sigma_-} = \tilde r_-$. Therefore, we have verified \eqref{eq:initial_cond_tilde_r} and \eqref{eq:lateral_cond_tilde_r}. This completes the proof of Proposition \ref{prop:existence_exp_grow}.
\end{proof}

\section{Proof of Theorem \ref{thm:main_result}}
\label{sec:proof}

Let $u_1\in H^1(Q)$ be an exponentially decaying CGO solution of the form \eqref{eq:CGO_v_old} satisfying the equation $\mathcal{L}_{c,g,q_1}^* u_1 = 0$ in $Q$, and let $u_2\in H_{\Box_{c,g}}(Q)$ be an exponentially growing CGO solution of the equation $\mathcal{L}_{c,g,q_2} u_2 = 0$ in $Q$ given by \eqref{eq:exp_grow_soln}  such that $u_2|_{t=0}=0$ and $\supp u_2|_\Sigma \subset U$. Due to the assumption $\mathcal{C}_{g,q_1} = \mathcal{C}_{g,q_2}$, by \cite[Proposition 3.1]{Kian_Oksanen}, there exists a function $v\in H_{\Box_{c,g}}(Q)$ that satisfies the equations $\mathcal{L}_{c,g,q_1} v = 0$ and 
\[
(u_2 - v)|_U =  (u_2 - v)|_{t=0}=(u_2 - v)|_{t=T}= \p_t(u_2 - v)|_{t=0}=\p_\nu (u_2 - v)|_V = 0.
\]
Then the function $u : = u_2 - v \in H_{\Box_{c,g}}(Q)$ is a solution to the equation
\begin{equation}
\label{eq:ibvp_difference}
\mathcal{L}_{c,g, q_1}u=qu_2 \quad \text{in} \quad Q, 
\quad 
u|_\Sigma= u|_{t=0}=u|_{t=T}= \p_tu|_{t=0}=\p_\nu u|_V = 0.
\end{equation}
Here and in what follows we used the notation $q:=q_1-q_2$.

After arguing similarly as in \cite[Section 6]{Kian_Oksanen}, we obtain the following integral identity
\begin{equation}
\label{eq:int_id_q}
\int_Q qu_2\overline{u_1}dV_gdt= \int_M c^{-1}\p_t u(T, x)\overline{u_1(T, x)}dV_g-  \int_{\Sigma\setminus V}\p_\nu u\overline{u_1}dS_g dt.
\end{equation}
We next substitute the CGO solutions \eqref{eq:CGO_v_old} and \eqref{eq:exp_grow_soln} into \eqref{eq:int_id_q} and pass to the limit $h\to 0$. 

To analyze the limit of the terms on right-hand side of \eqref{eq:int_id_q}, we have the following lemma, which states that both terms on the right-hand side of \eqref{eq:int_id_q} vanish as $h\to 0$. Its proof is same as the proof of \cite[Lemma 5.1]{Liu_Saksala_Yan}, which mainly relies on an application of a boundary Carleman estimate \cite[Theorem 4.1]{Kian_Oksanen}, as well as estimates \eqref{eq:estimate_v} and \eqref{eq:est_r1}. This lemma is the reason why we need the $H^1$-norm decay \eqref{eq:est_r1} for $r_1$.
\begin{lem}
\label{lem:rhs_est}
Let $u_1$ and $u$ be the functions described above. Then the following estimates hold as $h\to 0$:
\begin{equation}
\label{eq:est_int_id_rhs_1}
\int_M c^{-1}\p_t u(T, x)\overline{u_1(T, x)}dV_g = \cO(h^{1/2})
\end{equation}
and
\begin{equation}
\label{eq:est_int_id_rhs_2}
\int_{\Sigma\setminus V}\p_\nu u\overline{u_1}dS_g dt =\cO(h^{1/2}).
\end{equation}
\end{lem}

To investigate the left-hand side of the integral identity \eqref{eq:int_id_q}, we compute from the respective CGO solution \eqref{eq:CGO_v_old} and \eqref{eq:exp_grow_soln} for $u_1$ and $u_2$ that
\[
u_2\overline{u_1}=e^{2i\lambda (\beta t+x_1)}\confm(|v_s|^2+\overline{v_s}r_2+\overline{r_1}v_s+\overline{r_1}r_2).
\]
By 
estimates \eqref{eq:estimate_v}, \eqref{eq:est_r1}, \eqref{eq:est_r2}, 
and the Cauchy-Schwartz inequality, we obtain the estimate
\begin{equation}
\label{eq:quasimode_zero}
\bigg|\int_Q e^{2i\lambda (\beta t+x_1)} \confm (\overline{v_s}r_2 + \overline{r_1}v_s+\overline{r_1}r_2)dV_gdt\bigg|=\cO(h^{1/2}), \quad h\to 0.
\end{equation}

On the other hand, since $q_1, q_2 \in C(\overline{Q})$ and $q_1=q_2$ on the boundary $\p Q$, we may continuously extend $q$ by zero on $(\R^2\times M_0)\setminus Q$ and denote the extension by the same letter. Then from the equality $dV_g = c^{\frac{n}{2}}dV_{g_0}dx_1$, Fubini's theorem, the dominated convergence theorem, and the concentration property \eqref{eq:limit_prod_vw}, we obtain the following limit as $h\to 0$:
\begin{equation}
\label{eq:limit_prod_quasimodes}
\begin{aligned}
\int_Q e^{2i\lambda (\beta t+x_1)} \confm q|v_s|^2dV_gdt
&= \int_\R \int_\R \int_{M_0} e^{2i\lambda (\beta t+x_1)} c q|v_s|^2 dV_{g_0}dx_1dt
\\
&\to \int_{0}^{L} \int_\R \int_\R  e^{2i\lambda (\beta t+x_1)-2\sqrt{1-\beta^2} \lambda \tau} (cq)(t,x_1, \gamma(\tau))
dx_1dt d\tau.
\end{aligned}  
\end{equation}
Hence, by replacing $2\lambda$ with $\lambda$, we deduce from  \eqref{eq:int_id_q}, \eqref{eq:quasimode_zero}, \eqref{eq:limit_prod_quasimodes}, and Lemma \ref{lem:rhs_est} that the identity 
\begin{equation}
\label{eq:cgo_sub}
\int_{0}^{L} \int_\R \int_\R  e^{i\lambda (\beta t+x_1)-\sqrt{1-\beta^2} \lambda \tau} (cq)(t,x_1, \gamma(\tau)) dx_1dtd\tau =0
\end{equation}
holds for every non-tangential geodesic $\gamma$ in the transversal manifold $(M_0,g_0)$.

We are ready to utilize Assumption \ref{asu:inj}, the invertibility of the attenuated geodesic ray transform on $(M_0,g_0)$. To that end, we denote $\mathcal{F}_{(t, x_1)\to (\xi_1, \xi_2)}$ the Fourier transform in the two Euclidean variables $(t,x_1)$ and define
\begin{align*}
f(x', \beta, \lambda) &:= \int_\R\int_\R e^{i\lambda (\beta t+x_1)} (cq)(t, x_1, x')dx_1dt
=\mathcal{F}_{(t, x_1)\to (\xi_1, \xi_2)}(cq)|_{(\xi_1, \xi_2)=-\lambda(\beta, 1)},
\end{align*}
where $ x'\in M_0$, $\beta\in (\frac{1}{\sqrt{3}}, 1)$,  and $\lambda \in \R$. Since $q\in C(\overline{Q})$, the function $f(\cdot, \beta, \lambda)$ is continuous on $M_0$. As $\gamma$ is an arbitrary non-tangential geodesic, it follows from \eqref{eq:cgo_sub} that the following attenuated geodesic ray transform 
\begin{equation}
\label{eq:Fou_geo_trans}
I^{-\sqrt{1-\beta^2}\lambda}(f(\cdot,\beta, \lambda))(x, \xi)
=\int_0^{\tau_{\mathrm{exit}}(x, \xi)} e^{-\sqrt{1-\beta^2}\lambda \tau} f(\gamma_{x, \xi}(\tau), \beta, \lambda)d\tau, 
\quad (x, \xi) \in \p_-SM_0 \setminus \Gamma_-, 
\end{equation}
vanishes. 

By Assumption \ref{asu:inj}, there exists $\varepsilon>0$ such that $f(\gamma(\tau), \beta, \lambda)=0$ when $\sqrt{1-\beta^2}|\lambda|<\varepsilon$. Hence, there exist constants $\beta_0\in (\frac{1}{\sqrt 3}, 1)$, $\lambda_0>0$, and $\delta>0$ such that for every $(\lambda, \beta)\in \R^2$ satisfying $|\beta-\beta_0|$, $|\lambda-\lambda_0|<\delta$, and $\lambda\ne 0$, we have $\sqrt{1-\beta^2}|\lambda|<\varepsilon$. Thus, we see that $\mathcal{F}_{(t, x_1)\to (\xi_1, \xi_2)}(cq)=0$ in an open set of $\R^2$. Furthermore, since $q$ is compactly supported, we get from the Paley-Wiener theorem that $\cF(cq)$ is real analytic. Therefore, we conclude that $cq=0$ in $Q$. Finally, since $c$ is a positive function, we have $q=q_1-q_2=0$. This completes the proof of Theorem \ref{thm:main_result}.

\bibliographystyle{abbrv}
\bibliography{bibliography_potential}
\end{document}